\documentclass[11pt,reqno]{amsart}
\usepackage{hyperref}\hypersetup{colorlinks=true, citecolor=blue}
\usepackage{graphicx}
\usepackage{xfrac}
\usepackage{esint,amsfonts,amsmath,amssymb,epsfig,
mathrsfs,bm,accents,mathtools,array}
\usepackage{xfrac}

\numberwithin{equation}{section}

\newtheoremstyle{mytheorem}
{}
{}
{\it}
{\parindent}
{\bf}
{.}
{ }
{\thmnumber{#2.~}\thmname{#1}\thmnote{~\rm#3}}

\newtheoremstyle{myremark}
{}
{}
{\rm}
{\parindent}
{\bf}
{.}
{ }
{\thmnumber{#2.~}\thmname{#1}\thmnote{~\rm#3}}

\newtheoremstyle{myparagraph}
{}
{}
{\rm}
{\parindent}
{\bf}
{}
{ }
{\thmnumber{#2.~}\thmname{#1}\thmnote{#3}}

\theoremstyle{mytheorem}

\newtheorem{theorem}[subsubsection]{Theorem}
\newtheorem{lemma}[subsubsection]{Lemma}
\newtheorem{corollary}[subsubsection]{Corollary}
\newtheorem{proposition}[subsubsection]{Proposition}

\newtheorem{definition}[subsubsection]{Definition}

\theoremstyle{myremark}
\newtheorem{remark}[subsection]{Remark}
\newtheorem*{remark*}{Remark}

\theoremstyle{myparagraph}

\newtheorem*{parag*}{}

\makeatletter
\def\@secnumfont{\sc}
\def\section{\@startsection{section}{1}%
\z@{1.5\linespacing\@plus .2\linespacing}{.7\linespacing}%
{\normalfont\sc\centering}}
\makeatother

\makeatletter
\def\ps@headings{\ps@empty
 \def\@evenhead{%
  \setTrue{runhead}%
  \normalfont\footnotesize
  \rlap{\thepage}\hfil
  \def\thanks{\protect\thanks@warning}%
  \leftmark{}{}\hfil}%
 \def\@oddhead{%
  \setTrue{runhead}%
  \normalfont\footnotesize\hfil
  \def\thanks{\protect\thanks@warning}%
  \rightmark{}{}\hfil \llap{\thepage}}%
\let\@mkboth\markboth}
\makeatother
\makeatletter
\renewenvironment{proof}[1][\proofname]{\par
  \pushQED{\qed}%
  \normalfont \topsep6\p@\@plus6\p@\relax
  \trivlist
  \itemindent\normalparindent
  \item[\hskip\labelsep
    \scshape
    #1\@addpunct{.}]\ignorespaces
}{%
  \popQED\endtrivlist\@endpefalse
}
\providecommand{\proofname}{Proof}
\makeatother
\newcommand{\grad}{\nabla}

\newcommand{\N}{\mathbb{N}}

\newcommand{\R}{\mathbb{R}}
\newcommand{\Q}{\mathbb{Q}}

\newcommand{\Z}{\mathbb{Z}}

\renewcommand{\div}{\mathrm{div}}
\DeclareMathOperator*{\argmin}{arg\,min}

\DeclareMathOperator{\dist}{dist}

\def\loc{{\mathrm{loc}}}

\newcommand{\Ha}{\ensuremath{\mathcal{H}}}

\newcommand{\vh}{v^{(h)}}

\newcommand{\Per}{\mathrm{Per}}

\renewcommand{\d}{\mathrm{d}}

\newcommand{\sdF}{\mathrm{sd}_{ F}}
\newcommand{\sd}{\mathrm{sd}}
\newcommand{\Fh}{\mathcal F_h}
\newcommand{\Eh}{E^{(h)}}
\newcommand{\Ehk}{E^{(h_k)}}

\newcommand{\chih}{\chi^{(h)}}

\newcommand{\lambdah}{\lambda^{(h)}}

\newcommand{\eps}{\varepsilon}
\newcommand{\la}{\langle}
\newcommand{\ra}{\rangle}

\newcommand{\BRhbeta}{B_{Rh^\beta}(x_0)}
\newcommand{\Ebetah}{E^{(h),\beta}}

\newcommand{\tacka}{\, \cdot\,}

\newcommand{\de}{\partial}

\newcommand\res{\mathop{\hbox{\vrule height 7pt width .5pt depth 0pt
\vrule height .5pt width 6pt depth 0pt}}\nolimits}
\newcommand\weaks{{\stackrel{*}{\rightharpoonup}}\,}

\newcommand{\bC}{{\mathbf {C}}}

\begin{document}

	%
\pagestyle{empty}
\pagestyle{myheadings}
\markboth%
{\underline{\centerline{\hfill\footnotesize%
\textsc{L. Mugnai, C. Seis \& E. Spadaro}%
\vphantom{,}\hfill}}}%
{\underline{\centerline{\hfill\footnotesize%
\textsc{Volume-preserving Mean-curvature Flow}%
\vphantom{,}\hfill}}}

	%
\thispagestyle{empty}

~\vskip -1.1 cm

	%

\vspace{1.7 cm}

	%
{\Large\sl\centering
Global solutions to the volume-preserving\\
mean-curvature flow
\\
}

\vspace{.4 cm}

	%
\centerline{\sc Luca Mugnai, Christian Seis \& Emanuele Spadaro}

\vspace{.8 cm}

{\rightskip 1 cm
\leftskip 1 cm
\parindent 0 pt
\footnotesize

	%
{\sc Abstract.}
In this paper, we construct global distributional solutions to the volume-preserving mean-curvature flow using a variant of the time-discrete gradient flow approach proposed independently by Almgren, Taylor \& Wang \cite{ATW} and Luckhaus \& Sturzenhecker \cite{LuckhausSturzenhecker95}.

\par
}

%
%
\section{Introduction}
A family of open sets with smooth boundary
$\{ E_t\}_{0\leq t \leq T}$ in $\R^n$ is said to move
according to volume-preserving mean-curvature flow if the motion law,
expressed as an evolution equation for the boundaries $\partial E_t$, takes the form
\begin{equation}\label{eq:volMCF}
v\;=\; \la H\ra- H\quad\mbox{on }\partial E_t,
\end{equation}
for all $t\in[0,T]$. Here, at any point $x$ on $\partial E_t$, $v(x)$
denotes the velocity component normal to the boundary, in the direction of the outer
normal, $H(x)$ is the scalar mean curvature
(with the sign convention that $H$ is positive for balls, see the next section), and the brackets $\la\cdot\ra$
denote the average of a quantity over the boundary of $E_t$.

\medskip

It is immediately verified that the volume of the sets $E_t$
(i.e., its $n$-dimensional Lebesgue measure, denoted by $|E_t|$)
is indeed preserved under the smooth flow \eqref{eq:volMCF} because
\[
\frac{d}{dt} |E_t|\;=\; \int_{\partial E_t} v\, d\Ha^{n-1} \;\stackrel{\eqref{eq:volMCF}}{=}\; 0.
\]
And thus, upon rescaling variables, we may assume that $|E_t|=|E_0|=1$ for any $t\in[0,T]$. Moreover, the perimeter of 
the sets $E_t$ is decreasing because
\[
\frac{d}{dt}\Ha^{n-1}(\partial E_t)\;=\; \int_{\partial E_t} vH\, d\Ha^{n-1}\;\stackrel{\eqref{eq:volMCF}}{=}\; 
-\int_{\partial E_t} v^2\, d\Ha^{n-1}\;\le\;0.
\]
\medskip

During a typical evolution, a volume-preserving mean-curvature flow exhibits singularities of different kinds, even in the case of smooth initial data.
These singularities correspond to changes in the topology of the configuration
and include shrinkage of islands to points and disappearance, collisions
and merging of neighboring islands, pinch-offs etc\ldots\,
If the topology changes, the boundary of the evolving
set looses regularity and, as a consequence, the formulation \eqref{eq:volMCF} of the
evolution law is inadequate. 
The goal of the present work
is the construction of a notion of a weak solution to the volume-preserving mean-curvature flow that is global in time and thus overcomes these singular moments.

\medskip

Several solutions to volume-preserving mean-curvature 
flow have been proposed in the literature: existence and uniqueness of a global in time smooth solution and its 
convergence to a sphere is shown in \cite{Gage86,Huisken87} for smooth convex initial data and in 
\cite{EscherSimonett98,Li09} for initial data close to a sphere (for further related results see 
\cite{Andrews01,CR-M07,CR-M12} and the references therein).
In principle, these results also yield local in time existence and
uniqueness of smooth solutions. In \cite{PengMeriman} and \cite{RuuthWetton03}, the authors consider level-set and 
diffusion-generated solutions for the purpose of numerical studies.
In \cite{BCCN},
Bellettini, Caselles, Chambolle and Novaga construct solutions 
for the volume-preserving anisotropic mean-curvature (or anisotropic 
variant of (1.1)) for convex sets using a method similar to ours (that 
will be outlined below). In this paper, it is also shown that the 
solution (the so-called flat flow) is unique
and coincides with regular solutions when the latter are defined.
We also mention a mean field approximation approach to the volume preserving mean-curvature
flow as developed in \cite{BS, CHL, AA}.

\medskip

It is well-known that volume-preserving mean-curvature can be (formally) interpreted 
as the $L^2$-gradient flow of the perimeter functional for configurations with a fixed
volume, see, e.g., \cite[Sec.\ 2]{MS13}. This gradient flow structure, however, is for 
the purpose of well-posedness results impracticable, since the $L^2$ (geodesic) distance 
is degenerate in the sense that two well separated configurations may have zero 
$L^2$ distance \cite{MichorMumford06}. In the present manuscript, we follow the method 
proposed independently by Almgren, Taylor \& Wang \cite{ATW} and Luckhaus \& 
Sturzenhecker \cite{LuckhausSturzenhecker95} in the study of (forced) mean-curvature 
flows to bypass this difficulty. The authors consider an implicit time-discretization 
of the flow, which comes as a gradient flow of the perimeter functional with respect to 
a new non-degenerate distance function that approximates the $L^2$ distance.
The limiting time-continuous 
flow constructed with this method is usually referred to as the flat 
flow. The main 
difference between the present work and \cite{ATW,LuckhausSturzenhecker95} relies on 
the non-locality of the volume-preserving mean-curvature flow. As an immediate 
consequence, there is no maximum-principle available for \eqref{eq:volMCF}.
Also related to this aspect, there is the problem of proving the consistency of the scheme, i.e.~the
coincidence of the flat volume-preserving mean-curvature flow with the smooth one when the latter
exists. Under some assumption on the Lagrange multiplier of the flow,
the consistency can be inferred from the arguments in \cite{ATW}, but we do not know
if these conditions are generally satisfied and we do not discuss further the problem
of the consistency in this paper.
A more detailed discussion on the different features of the flows in \cite{ATW,LuckhausSturzenhecker95} 
and the one considered in the present manuscript will follow in 
Section~\ref{s:3} below.

\medskip

We conclude this subsection with a short discussion on the background of this 
evolution. Volume-preserving mean-curvature flow can be considered as a simplified
model for attachment-limited kinetics, and as such it plays an important role in
the study of solidification processes, where solid islands grow in an
undercooled liquid of the same substance. In such situations, solid particles
melt at high-curvature regions and simultaneously precipitate at
low-curvature regions, while the total mass of the solid remains essentially
constant \cite{Wagner61,CRCT95,Tarshis}. In this way, the total surface area
of solid islands is decreasing, and thus, this process leads to the growth of
larger islands at the expense of smaller ones: a phenomenon called coarsening
\cite{MS13}. In general, solidification processes are mathematically often
modelled by Mullins--Sekerka equations (or a Stefan problem), where the
Gibbs--Thompson relation is modified by a kinetic drift term \cite{LS61,Gurtin93},
and their phase-field counterparts respectively \cite{Caginalp89}. This model allows for\
both attachment kinetics (kinetic drift) and bulk diffusion (Mullins--Sekerka/Stefan). 
It turns out that attachment kinetics is the relevant mass transport mechanism in 
earlier stages of the evolution while bulk diffusion predominates the later stages 
\cite{DaiNiethammerPego}. In a certain sense, volume-preserving mean-curvature flow 
naturally arises as the singular limit of this more general solidification model in the 
regime of vanishing bulk diffusion. More recently, variants of volume-preserving mean 
curvature flow were also applied in the context of shape recovery in image processing 
\cite{CapuzzoFinziMarch02}.

\medskip

The article is organized as follows: in Section~\ref{s:2} we fix the notation and
state the main results of the paper, which are then proved
in Sections~\ref{s:3} and \ref{s:4} and are 
the existence of flat volume-preserving mean-curvature flows
and the existence of distributional solutions, respectively.

%
%
\section*{Acknowledgements}
C.S. acknowledges the kind hospitality of the Max-Planck-Institut f\"ur Mathematik in den Naturwissenschaften in 
Leipzig, Germany, where part of this work was performed.

%
%
\section{Statements of the main results}\label{s:2}
\subsection{Notation}
For any Lebesgue measurable set $E \subset \R^n$, we denote by $|E|$ the 
$n$-dimensional Lebesgue measure of $E$ and by $\chi_E$ the characteristic 
function of $E$, i.e.~$\chi_E(x) = 0$ if $x \notin E$ and $\chi_E(x) =1$ if $x \in E$.
The perimeter of $E$ in an open set $\Omega\subset\R^n$ is 
defined as
\[
\Per(E,\Omega):=\sup\left\{\int_E\div \varphi\, dx:\: \varphi \in 
C_c^1(\Omega;\R^n)\mbox{ with }\sup_{\Omega}|\varphi|\le1\right\},
\]
and we write $\Per(E):=\Per(E,\R^n)$. If the latter quantity is finite, we will
call $E$ a set of finite perimeter. In the case that $E$ is an open set with $\de E$
of class $C^1$, we simply have  $\Per(E,\Omega)=\Ha^{n-1}(\partial E\cap\Omega)$ 
and $\Per(E)=\Ha^{n-1}(\partial E)$. The reduced boundary of a set of finite perimeter 
$E$ is denoted by $\partial^* E$, cp.\ \cite[Sec.\ 5.7]{EvansGariepy}, and 
for the unit outer normal to $E$ we write $\nu_{E}$. The tangential divergence 
of a vector field $\Psi\in C^1(\R^n;\R^n)$ with respect to $\partial E$ is 
defined by $\div_{\partial E}\Psi := \div \Psi - \nu_E \cdot\grad\Psi\nu_E$. We 
say that a set of finite perimeter $E$ has a (generalized) mean-curvature $H_E\in 
L^1(\partial^* E,d\Ha^{n-1})$ provided that
\begin{equation}\label{e:generalized mean-curv}
\int_{\partial^*E} \div_{\partial E}\Psi\, d\Ha^{n-1} =  \int_{\partial^* E} 
\Psi\cdot \nu_E H_E\, d\Ha^{n-1}\quad\mbox{for all }\Psi\in C_c^{1}(\R^n;\R^n).
\end{equation}
Observe that with this sign convention it is $H_{B_R}\equiv\frac{n-1}R$.

\medskip

We write $\sd_{ F}$ for the signed distance from a Lebesgue measurable set $F$, with the 
convention that $\sd_{ F}$ is negative inside $F$, i.e.,
\[
 \sd_{ F}(x)\;=\;\begin{cases} \dist(x, F) &\mbox{for }x\in F^c,
 \\ -\dist(x,F^c)&\mbox{for }x\in F,
\end{cases}
\]
Here, $F^c := \R^n \setminus F$ denotes the complement set of $F$,
and the distance from a set $F$ is by definition
\[
\dist(x,F) = \inf_{y \in F} |x-y|.
\]
We will sometimes use the notation $\d_F=|\sd_F|$.

\medskip

By $[t]$ we denote the integer part of a real number $t$, that is the biggest 
integer $m$ such that $m\leq t$.

\medskip

Finally, we denote by $c_n, C_n$ positive constants that depend on the space dimension
only. Moreover, $c_{n,0}$, $C_{n,0}$, and $C_{n,0,T}$ are 
constants that may additionally depend on the initial data or the time $T>0$. During 
the computations, the value of these constants may change from line to line.
However, for the sake of clarity we need to keep track of the dimensional constant
in Proposition~\ref{p:Linfty}, therefore we make an exception to
the above convention and denote it by $\gamma_n$.
The volume of the $n$-dimensional unit ball will be denoted by $\omega_n$, and thus its surface 
area is $n\omega_n$.

\subsection{Approximate solutions}
In this paper we introduce a notion of global flat solution to the volume-preserving
mean-curvature flow which is based on the implicit time discretization of
\eqref{eq:volMCF} in the spirit of Almgren, Taylor \& Wang \cite{ATW}
and Luckhaus \& Sturzenhecker \cite{LuckhausSturzenhecker95}.
That is, we consider a time-discrete gradient flow for the perimeter functional.
For this purpose, we define
\[
\Fh(E,F):=\Per(E) + \frac1{h} \int_E \sd_{ F}\, dx + \frac{1}{\sqrt{h}}| |E|-1|,
\]
for any two sets of finite perimeter $E$ and $F$ in $\R^n$. Here, $h$ is a 
positive small number that plays the role of the time step of approximate solutions.
The second term in the above 
functional 
approximates the degenerate $L^2$ geodesic distance on the configuration space 
of hypersurfaces. The functional differs from the one considered in the original 
papers \cite{ATW,LuckhausSturzenhecker95} only in the last term: a weak 
penalization that favors unit-volume of minimizing sets.

\begin{definition}\label{D2}
Let $E_0$ be a set of finite perimeter with $|E_0|=1$, and $h>0$.
Let $\{\Eh_{kh}\}_{k\in\N}$ be a sequence of sets defined iteratively by  
\[
\Eh_0=E_0\quad\mbox{and}\quad \Eh_{kh}\in\argmin_{E \subset \R^n}
\left\{\Fh(E,\Eh_{(k-1)h})\right\}\quad\mbox{for }k\ge1.
\]
We furthermore define
\[
\Eh_t:=\Eh_{kh}\quad\mbox{for any }t\in[kh,(k+1)h),
\]
and call $\{\Eh_t\}_{t\geq 0}$ an \textup{approximate flat solution}
to the volume-preserving mean-curvature flow with initial datum $E_0$.
\end{definition}

The existence of minimizers $\Eh_{kh}$ and thus the existence of an approximate 
flat solution is guaranteed by Lemma \ref{L1} below. Incorporating the volume 
constraint in a soft way into the energy functional rather than imposing a hard 
constraint on the admissible sets has the advantage that we are free to chose 
arbitrary competitors,
most notably in the derivation of density estimates.
Thanks to the penalizing factor $1/\sqrt h$, the constraint becomes active in 
the limit $h\downarrow0$. Even more can be shown: the number of time steps in 
which approximate solutions violate the volume constraint $|\Eh_t|=1$ can be 
bounded uniformly in $h$, cf.\ Corollary \ref{prop:soft2hard} below. A similar 
functional including a soft volume constraint was recently considered by Goldman 
\& Novaga \cite{GoldmanNovaga12} in the study of a prescribed curvature problem.

\subsection{Main results}
We can now state our main results. 
The first one is a convergence result for approximate solutions.

\begin{theorem}[(Existence of flat flows)]\label{theo:conv-L1}
Let $E_0$ be a bounded set of finite perimeter with $|E_0|=1$ and, 
for any $h>0$, let $\{\Eh_t\}_{t \geq 0}$ be an approximate flat solution
to the volume-preserving mean-curvature flow with initial datum $E_0$.
Then, there exists a family of sets of finite perimeter
$\{E_t\}_{t\geq 0}$
and a subsequence $h_k\downarrow 0$
such that
\[
|E_t^{(h_k)} \Delta E_t| \to 0 \quad \text{for a.e. } t \in [0, +\infty),
\]
and, for every $0\le s\le t$, 
\begin{gather}
|E_t| = 1,\nonumber\\
\vert E_t\Delta E_s\vert\leq C_{n,0}\vert t-s\vert^{1/2},
\nonumber\\
\Per (E_t)\leq\Per(E_s).
\nonumber
\end{gather}
\end{theorem}

\medskip

Our next statement is the existence of a distributional solution
in the sense of Luckhaus \& Sturzenhecker \cite{LuckhausSturzenhecker95}
to the volume-preserving mean-curvature flow under the hypothesis
that the perimeters of the approximate solutions converge to the perimeter of 
the limiting solutions identified in the previous theorem.

\begin{theorem}[(Existence of distributional solutions)]\label{theo:frenata}
Suppose that $n\leq 7$. Let $(h_k)_{k\in \N}$ and $\{E_t\}_{t\geq 0}$ be as
in Theorem~\ref{theo:conv-L1}. For any $T>0$, if
\begin{equation}\label{eq:ass-LS}
\lim_{k\to\infty}\int_0^T\Per(\Ehk_t)\,dt=\int_0^T\Per(E_t)\,dt,
\end{equation}
then $\{E_t\}_{0 \leq t<T}$ is a distributional solution to the volume-preserving 
mean-curvature flow with initial datum $E_0$ in the following sense:
\begin{enumerate}
\item for almost every $t\in [0,T)$ the set $E_t$ has (generalized)
mean curvature in the sense of \eqref{e:generalized mean-curv}
satisfying
\begin{equation}\label{e:gen mean-curv}
\int_0^T\hspace{-0.2cm} \int_{\de^* E_t}|H_{E_t}|^2 < +\infty;
\end{equation}

\item there exists $v:\R^n \times (0,T) \to \R$
with $v(\cdot, t)\vert_{\de^* E_t} \in L^2_0( \partial^*E_t,d\Ha^{n-1})$
for a.e.~$t \in (0,T)$
and
$\int_0^T \int_{\de^* E_t} v^2 d\Ha^{n-1}dt <+\infty$
such that
\begin{gather}
-\int_0^T\hspace{-0.2cm}\int_{\partial^* E_t}v\,\phi\,d\Ha^{n-1}dt =
\int_0^T\hspace{-0.2cm}\int_{\partial^* E_t}\left(H_{E_t}\,\phi-\lambda\,\phi\right)d\Ha^{n-1}dt,\label{e:volMCF weak1}\\
\int_0^T\hspace{-0.2cm}\int_{E_t}\partial_t\phi\,dxdt+\int_{E_0}\phi(0,\tacka)\,dx=
-\int_0^T\hspace{-0.2cm}\int_{\partial^* E_t}v\,\phi\,d\Ha^{n-1}dt,\label{e:volMCF weak2}
\end{gather}
for every $\phi\in C_c^1([0,\infty)\times\R^n)$, where 
\begin{equation}\label{e:lagrange multip}
\lambda(t):=\frac{1}{\Ha^{n-1}(\partial^*E_t)} \int_{\partial^* E_t} H_{E_t}\,d\Ha^{n-1}
\quad\text{for a.e.~$t\in [0,T)$.}
\end{equation}
\end{enumerate}
\end{theorem}

In the second part of the theorem, $L^2_0$ is the set of all $L^2$ functions with zero mean.

Note that \eqref{e:volMCF weak1} is a weak formulation of \eqref{eq:volMCF},
while \eqref{e:volMCF weak2} establishes the link between $v$ and the velocity
of the boundaries of $E_t$.
It is straightforward to check that smooth solutions of \eqref{eq:volMCF} 
satisfy \eqref{e:volMCF weak1} and \eqref{e:volMCF weak2}.

The restriction on the dimension $n\leq 7$ is technical and is needed
in the proof of Corollary~\ref{lem:rosticini} where the Bernstein theorem
for minimal surfaces is exploited.

%
%
%

%
%
\section{Flat volume-preserving mean-curvature flows}\label{s:3}
In this section we prove the first main result in Theorem~\ref{theo:conv-L1}.
We follow quite closely
Luckhaus \& Sturzenhecker \cite{LuckhausSturzenhecker95}, providing all
the details for the readers' convenience.

\subsection{Existence of approximate solutions}
We start remarking that
\begin{equation}\label{1}
 \int_E\sd_{ F}\, dx = \int_{E\Delta F}\d_{ F}\, dx - \int_{F} \d_{ F}\, 
dx.
\end{equation}

The existence of the approximate flat solutions is guaranteed by the following 
lemma.
\begin{lemma}\label{L1}
Let $F \subset \R^n$ be a bounded set of finite perimeter. For every $h>0$, there 
exists a minimizer $E$ of $\Fh(\tacka, F)$ and, moreover, $E$ satisfies
the discrete dissipation inequality
\begin{equation}\label{2}
\Per(E) + \frac1h \int_{E\Delta F}\d_{ F}\, dx + \frac{1}{\sqrt{h}}\left||E|-1\right|
\;\le\; \Per(F) + \frac{1}{\sqrt{h}}\left||F|-1\right|.
\end{equation}
\end{lemma}

\begin{proof}
Since $F$ is an admissible competitor, we obtain by \eqref{1} that
\begin{equation}\label{e:dissipation}
0\;<\;\inf_{\tilde E} \Fh(\tilde E,F)+\frac1h\int_{F} \d_{F}\, dx\;\le\;\Per(F) + \frac{1}{\sqrt{h}}\left||F|-1\right|\;<\;\infty.
\end{equation}
Let $\{E_{\nu}\}_{\nu\in\N}$ denote a minimizing sequence of $\Fh(\tacka, F)$.
Without loss of generality we may assume that $E_{\nu}\subset\subset B_R$
for a suitable $R>0$.
Since $\{\chi_{E_{\nu}}\}_{\nu\in\N}$ is bounded in $BV(B_R)$, there exists a
subsequence (not relabeled) that converges weakly to a function $\chi$ in $BV(B_R)$, and thus strongly in $L^1(B_R)$.
In particular, $\chi$ is the characteristic function of some set of finite perimeter $E$. Since $\chi_{\tilde 
E}\mapsto \int_{\tilde E}  \sd_F\, dx$ is continuous 
and the perimeter is lower semi-continuous with respect to $L^1$ convergence, it follows that
\[
\Fh(E,F)\;\le\; \liminf_{\nu\uparrow\infty} \Fh(E_{\nu},F) = \inf_{\tilde E \subset\R^n} \Fh(\tilde E,F).
\]
Therefore, $E$ minimizes $\Fh(\tacka,F)$ and
\eqref{2} follows from \eqref{e:dissipation}.
\end{proof}

By standard results on minimal surfaces (see \cite{Maggi}), it holds that
the minimizers $E$ of $\Fh(\tacka,F)$ can be chosen to be closed subsets with
$\de E$ of class $C^{1,\alpha}$ up to a (relatively closed) singular set of dimension
at most $n-7$. Using the Euler--Lagrange equation for $\Fh(\tacka,F)$,
one can also show that the regular part of the boundary $\de E$ is actually
$C^{2,\kappa}$  (cp.~Lemma~\ref{lem:EL}).

\subsection{$L^\infty$ and $L^1$-estimates}
Our next statement gives a uniform bound on the distance between the boundary of the minimizing set and the boundary of 
the reference set.

\begin{proposition}[($L^\infty$-estimate)]\label{p:Linfty}
There exists a dimensional constant $\gamma_n>0$ with the following property.
Let $F \subset \R^n$ be a bounded set of finite perimeter
and let $E$ be a minimizer of $\Fh(\tacka,F)$. Then,
\begin{equation}\label{e:Linfty}
\sup_{E\Delta F} \d_{F}\leq \gamma_{n}\sqrt{h}.
\end{equation}
\end{proposition}

\begin{proof}
The proof of this proposition is based on the density estimates for one-side minimizing set
which for readers' convenience we prove in the Appendix~\ref{s:a}.
We claim indeed that the statement holds with
$$
\gamma_{n} = \max\Big\{3, \frac{4n\omega_n}{c_n}\Big\},
$$ 
where $c_n$ is the dimensional constant in Lemma~\ref{l.one side}.
The argument is by contradiction. Let 
$c>\max\left\{3, \frac{4n\omega_n}{c_n}\right\}$ and let
$x_0\in F\Delta E$ contradict \eqref{e:Linfty} with $\gamma_n$ replaced by $c$.
Without loss of generality, we can assume that $x_0 \in F\setminus E$: the other case
is at all analogous. 
We then have that
\begin{equation}\label{4}
\sdF(x_0) < - c\, \sqrt{h}.
\end{equation}
Then any ball $B_r(x_0)$ of radius $r\le\frac{c\sqrt{h}}2$ 
is contained in $F$. By the minimality of $E$, we have
$\Fh(E,F)\le \Fh(E \cup B_r(x_0),F)$, and thus
\begin{equation}\label{e:estimate energy}
\Per(E) \leq \Per(E\cup B_r(x_0)) + 
\frac1h \int_{B_r(x_0)\setminus E} \sdF\, dx + \frac{1}{\sqrt{h}}|B_r(x_0) \setminus E|.
\end{equation}
We use \eqref{4} and $r\le\frac{c\sqrt{h}}2$ to infer that
\begin{equation}\label{e:estimate volume}
\frac1h\int_{B_r(x_0) \setminus E} \sdF\, dx < - \frac{c}{2\sqrt{h}} |B_r(x_0) \setminus E|.
\end{equation}
Then \eqref{e:estimate energy} and \eqref{e:estimate volume} yield
\begin{align}
\Per(E) &\le \Per (E \cup B_r(x_0)) - h^{-\sfrac12}
\left(\frac{c}{2} -1\right) |B_r(x_0) \setminus E|.\label{e.one side min1}
\end{align}
By assumption $c>3$ and we can apply Lemma \ref{l.one side} with $\mu =0$ 
and obtain
\begin{equation}\label{5}
|B_r(x_0) \setminus E| \geq c_n r^n \quad \mbox{for a.e. }0<r <\frac{c\sqrt{h}}2.
\end{equation}
On the other hand, from \eqref{e.one side min1} we deduce also that
for a.e.~$0<r <\frac{c\sqrt{h}}2$
\begin{align}\label{e.ball minus E}
h^{-\sfrac12}
\left(\frac{c}{2} -1\right) |B_r(x_0) \setminus E| & \leq 
\Per (E \cup B_r(x_0))  - \Per (E)\notag\\
&\leq \Ha^{n-1}(\de B_r(x_0) 
\setminus E) \leq n\,\omega_n\,r^{n-1}.
\end{align}
Combining \eqref{5} and \eqref{e.ball minus E}, we get that
\[
c_n r^n \leq |B_r(x_0) \setminus E|
\leq n\,\omega_n\,\left(\frac{c}{2} -1\right)^{-1}
\sqrt{h}\, r^{n-1},
\]
for almost all $0<r <\frac{c\sqrt{h}}2$, which gives the desired 
contradiction to the choice of $c$ as soon as $r \uparrow \frac{c\sqrt{h}}{2}$.
\end{proof}

The following density estimates are now an immediate consequence.

\begin{corollary}\label{c:density}
Let $F \subset \R^n$ be a bounded set of finite perimeter
and let $E$ be a minimizer of $\Fh(\tacka,F)$. Then, for every $r \in (0, 
\gamma_n\,\sqrt{h})$ and for every $x_0 \in \de E$, it holds
\begin{gather}
\min\big\{|B_r(x_0)\setminus E|,|E\cap B_r(x_0)|\big\} \geq c_n\, 
r^n,\label{e:density vol}\\
c_n r^{n-1}\le \Per(E , B_{r}(x_0)) \leq C_n\,r^{n-1}.\label{e:density area}
\end{gather}
\end{corollary}

\begin{proof}
Since $E$ is a minimizer of $\Fh(\tacka,F)$, for any $x_0\in \de E$, it holds 
that $\Fh(E,F)\le \Fh(E\setminus B_{r}(x_0),F)$, which implies
\[
\Per(E) + \frac1h \int_{E\cap B_{r}(x_0)} \sdF\, dx\leq  
\Per(E\setminus 
B_{r}(x_0)) + \frac{1}{\sqrt{h}}|E\cap B_{r}(x_0)|.
\]
Estimating the second term via Proposition~\ref{p:Linfty}, we obtain
\begin{equation}
\label{6bis}
\Per(E) \leq \Per(E\setminus B_{r}(x_0)) + \frac{C_n  }{\sqrt{h}} |E\cap 
B_{r}(x_0)|.
\end{equation}
A similar analysis shows that
\[
 \Per(E) \leq \Per(E\cup B_{r}(x_0)) + \frac{C_n  }{\sqrt{h}} 
|B_{r}(x_0)\setminus E|.
\]
Therefore, by Lemma \ref{l.one side} we deduce (by possibly redefining $c_n$)
\[
\min\big\{ |E\cap B_{r}(x_0)|, |B_{r}(x_0)\setminus E|\big\} \geq 
c_n\,r^n \quad\forall\;0<r \leq \gamma_n\sqrt{h}.
\]
The first inequality in \eqref{e:density area} is now an immediate consequence of the relative 
isoperimetric inequality (cf.~\cite[Cor.\ 1.29]{Giusti}). For the second inequality, we rewrite \eqref{6bis} as
\[
\Ha^{n-1}(\partial E\cap B_r(x_0)) \leq \Ha^{n-1}(\partial B_r(x_0) \cap E) + \frac{C_n}{\sqrt{h}} |E\cap B_r(x_0)|.
\]
Since $r< \gamma_n \sqrt{h}$, the upper bound is obvious.
\end{proof}

Next we prove an estimate on the volume of the symmetric difference of
two consecutive sets of the approximate solutions.

\begin{proposition}[($L^1$-estimate)]\label{p:L1}
Let $F \subset \R^n$ be a bounded set of finite perimeter
and let $E$ be a minimizer of $\Fh(\tacka,F)$. Then,
\begin{equation}\label{9}
 |E\Delta F|\;\le\;  C_n\left(\ell \,\Per(E) + \frac1{\ell} \int_{E\Delta 
F} \d_{F}\, dx\right) \quad\forall \; \ell\leq  \gamma_n\,\sqrt{h}.
\end{equation}
\end{proposition}

\begin{proof}

In order to estimate $E\Delta F$, we split it into two parts:
\[
|E\Delta F| \leq |\{x\in E\Delta F:\: \d_{F}(x)\le \ell\}| + 
|\{x\in E\Delta F:\: \d_{F}(x)\ge \ell\}|.
\]
The second term is easily estimated by
\[
 |\{x\in E\Delta F:\: \d_{F}(x)\geq \ell\}|\leq \frac1{\ell} 
\int_{E\Delta F} \d_{F}(x)\, dx.
\]
To estimate the first term, we use a simple covering argument
to find a collection of disjoint balls $\{B_{\ell}(x_i)\}_{i \in I}$ with 
$x_i\in\partial^*E$ and $I \subset \N$ a finite set such that 
$\de^* E \subset \cup_{i \in I}B_{2\ell}(x_i)$.
Note that by \eqref{e:density vol} and the relative isoperimetric inequality
(cf.~\cite[Cor.\ 1.29]{Giusti}) we have for every $i\in I$
\begin{align*}
|B_{3\ell}(x_i)| &
\stackrel{\eqref{e:density vol}}{\leq} C_n \min\{ |E\cap 
B_{\ell}(x_i)|, |B_{\ell}(x_i)\setminus E|\}\\
&\leq  C_n\ell\, \Ha^{n-1} (\partial^* E\cap B_{\ell}(x_i)).
\end{align*}
Note finally that the set $\{x\in E\Delta F:\: \d_{F}(x)\le \ell\}$ 
is covered by $\{B_{3\ell}(x_i)\}_{i\in I}$.
Summing over $i$ and the choice of the balls $\{B_{\ell}(x_i)\}_{i\in I}$ yields
\begin{align*}
 |\{x\in E\Delta F:\: \d_{ F}(x)\le \ell\}|&\leq \sum_{i \in 
I} |B_{3\ell}(x_i)|\\
&\leq C_{n}\, \ell \sum_{i\in I} \Ha^{n-1}(\partial^* E\cap B_{\ell}(x_i))\\
&\leq C_{n}\, \ell\,\Per(E). \qedhere
\end{align*}
\end{proof}

\subsection{H\"older continuity in time}

As an immediate consequence of the discrete dissipation inequality \eqref{2}, 
we remark that 
\begin{multline}\label{10}
\Per(\Eh_t) + \frac1h\int_{\Eh_t\Delta \Eh_{t-h}} \d_{\Eh_{t-h}}\, 
dx + \frac{1}{\sqrt{h}}\big\vert|\Eh_t|-1\big\vert\\
\leq \Per(\Eh_{t-h})+ 
\frac{1}{\sqrt{h}}\big\vert|\Eh_{t-h}|-1\big\vert
\quad\forall\; t\in [h,+\infty),
\end{multline}
and, by iterating \eqref{10} and using $|E_0|=1$,
\begin{gather}
\Per(\Eh_{t}) \leq \Per(E_0) \quad \forall\; t\geq0,\label{e:per decreasing}\\
\frac{1}{\sqrt{h}}\big\vert|\Eh_{t}|-1\big\vert\leq \Per(E_0) \quad \forall\; 
t\geq0,\label{e:volume converging}\\
\int_h^T \int_{\Eh_{t} \Delta \Eh_{t-h}} \frac{\d_{\Eh_{t-h}}}{h}\,dx \leq \Per(E_0),\label{e:extra1}
\end{gather}
for every $T>h$. Similarly, using Proposition~\ref{p:L1} with $\ell = h < \gamma_n \sqrt{h}$,
we also get
\begin{align}\label{e:extra2}
\int_h^T|\Eh_{t} \Delta \Eh_{t-h}| & \leq C_n \sum_{k=1}^{[\sfrac{T}{h}]} \left(h\,\Per(\Eh_{kh})+
\int_{\Eh_{kh} \Delta \Eh_{(k-1)h}}\frac{\d_{\Eh_{t-h}}}{h}\,dx \right)\notag\\
& \leq C_n\, (T +1)\,\Per(E_0).
\end{align}

\begin{proposition}[($C^{1/2}$ regularity in time)]\label{p:hoelder}
Let $h\le 1$ and let $\{E_t\}_{t \geq 0}$ be an approximate flat flow. 
Then it holds
\[
\vert\Eh_t\Delta\Eh_{s}\vert\;\le\;C_{n,0}\vert t-s\vert^{1/2}
\quad\forall\; 0\leq t\leq s<+\infty.
\]
\end{proposition}

\begin{proof}
Clearly it is enough to consider the case $s-t \geq {h}$.
Let $j\in \N$ and $k\in \N\setminus \{0\}$ be such that $t\in[jh, (j+1)h)$ 
and $s\in[(j+k)h, (j+k+1)h)$. Then, we can use Proposition~\ref{p:L1} with 
$\ell =\gamma_n \sfrac{h}{|t-s|^{\sfrac12}}$ (note that $ \ell \leq\gamma_n\sqrt h$ by the 
assumption $s-t \geq {h}$) and \eqref{e:per decreasing}, and estimate in the following way:
\begin{align}
\vert\Eh_t\Delta\Eh_{s}\vert &\leq  \sum_{m=1}^k  |\Eh_{(j+m)h}\Delta 
\Eh_{(j+m-1)h}|\notag\\
& \leq C_n  \sum_{m=1}^k \frac{h}{|t-s|^{1/2}} \Per(\Eh_{(j+m)h}) \notag\\
& \quad+ C_n  \sum_{m=1}^k  \frac{|t-s|^{1/2}}{h} \int_{\Eh_{(j+m)h}\Delta 
\Eh_{(j+m-1)h}} 
\mathrm{d}_{\Eh_{(j+m-1)h}}\, dx.\notag
\end{align}
By using \eqref{10} we estimate the sum above by 
\begin{align}
\vert\Eh_t\Delta\Eh_{s}\vert
& \leq C_n  \sum_{m=1}^k \frac{h}{|t-s|^{1/2}}\Per(E_0) \notag\\
&\quad
+ C_n  \sum_{m=1}^k |t-s|^{\sfrac12}\Big(\Per(\Eh_{(j+m-1)h}) - 
\Per(\Eh_{(j+m)h})\Big) \notag\\
& \quad+ C_n  \sum_{m=1}^k  \frac{|t-s|^{1/2}}{\sqrt{h}} 
\Big(  
\big\vert|\Eh_{(j+m-1)h}|-1\big\vert - \big\vert |\Eh_{(j+m)h}|-1\big\vert
\Big)\notag\\
& \leq C_n\frac{kh}{|t-s|^{1/2}}\Per(E_0)
+ C_n  |t-s|^{\sfrac12}\big(\Per(E_t) - 
\Per(E_s)\big)\notag\\
&\quad+ C_n  \frac{|t-s|^{1/2}}{\sqrt{h}} 
\Big(  
\big\vert|E_t|-1\big\vert - \big\vert|E_s|-1\big\vert
\Big).
\end{align}
Therefore, by \eqref{e:per decreasing} and \eqref{e:volume converging}, we get
\begin{align}
\vert\Eh_t\Delta\Eh_{s}\vert
& \leq  C_n\,|t-s|^{1/2}\Per(E_0),
\end{align}
where we used $kh\leq |t-s| + h\leq 2|t-s|$, thus concluding the proof of 
the proposition.
\end{proof}

\subsection{First variations and first consequences}
We now introduce the time-discrete normal velocity: for all $t\geq 0$ and $x 
\in \R^n$, we set
\[
\vh(t,x):=
\begin{cases}
\frac{1}{h}\sd_{\Eh_{t-h}}(x)& \mbox{for }t\in [h,+\infty),\\
0&\mbox{for }t\in [0,h).
\end{cases}
\]

\begin{lemma}[($L^2$-bound on the velocity)]\label{Lem:vel-L^2}
Let $\{\Eh_t\}_{t \geq 0}$ be an approximate flat flow.
Then  it holds
\begin{equation}\label{eq:L2-vh}
\int_0^{\infty}\int_{\partial \Eh_t}(\vh)^2\,d\Ha^{n-1}dt \leq C_{n,0}.
\end{equation}
\end{lemma}

\begin{proof}
We fix $t\in[h,+\infty)$ and consider for every $\ell\in\Z$ with $2^\ell \leq 
\sfrac{\gamma_n}{\sqrt{h}}$ the sets
$$
K(\ell):=\{x\in\R^d:~2^{\ell}<\vert \vh (t,x)\vert\leq 2^{\ell+1}\},
$$
so that $\R^n= \bigcup_{\ell} K(\ell)$.
It follows from $2^{\ell-1} h \leq \gamma_n\sqrt{h}/2$ and from 
Corollary~\ref{c:density} that, for every $x \in \de \Eh_t$,
\begin{gather}
|\Eh_t\cap B_{2^{\ell-1}h}(x)|\;\ge\; c_n 
\left(2^{\ell-1}h\right)^n,\label{17}\\
\Ha^{n-1}(\partial\Eh_t\cap B_{2^{\ell-1}h}(x))\;\le\; C_n \left(2^{\ell 
-1}h\right)^{n-1}.\label{18}
\end{gather}
Using $2^{\ell-1}\le |\vh 
(t,y)| \le 4 \cdot2^{\ell-1}$ for all $y\in B_{2^{\ell-1} h}(x)$ with 
$x\in\partial \Eh_t\cap K(\ell)$, we obtain for every  $x\in\partial \Eh_t\cap 
K(\ell)$
\begin{gather*}
\int_{B_{2^{\ell-1} h}(x)\cap(\Eh_t\Delta \Eh_{t-h})}\vert \vh \vert\,dy
\stackrel{\eqref{17}}{\geq}
c_n \,2^{\ell-1}\left(2^{\ell-1}h\right)^n,\\
\int_{B_{2^{\ell-1}h }(x)\cap\partial \Eh_t}( \vh 
)^2\,d\Ha^{n-1}\;\stackrel{\eqref{18}}{\le}\; C_n (2^{\ell-1})^2 
\left(2^{\ell-1} h\right)^{n-1}.
\end{gather*}
Hence, combining these two estimates, we have
$$
\int_{B_{2^{\ell-1} h}(x)\cap\partial \Eh_t}( \vh )^2\,d\Ha^{n-1}\;\le\; 
\frac{C_n}{ h}\int_{B_{2^{\ell-1} h}(x)\cap(\Eh_t\Delta \Eh_{(k-1)h})}\vert \vh 
\vert\,dy.
$$
Now, by a simple application of Besicovitch's covering theorem 
\cite[Ch.\ 1.5.2]{EvansGariepy} to $\{ B_{2^{\ell-1} h}(x):~x\in\partial 
\Eh_t\cap K(\ell)\}$, we obtain 
\begin{equation}\label{eq:nando}
\int_{\de \Eh_t \cap K(\ell)} \big(\vh\big)^2 \, d\Ha^{n-1} \leq \frac{C_n}{h} 
\int_{(\Eh_t \Delta \Eh_{t-h})\cap \{2^{\ell-2}\leq|\vh|\leq 
2^{\ell+2}\}}|\vh|\,dx.
\end{equation}
Finally, summing up over $\ell\in\Z$ with $2^{\ell}\le \frac{\gamma_n}{\sqrt 
h}$ in \eqref{eq:nando} yields
\[
\int_{\partial \Eh_t}( \vh )^2\,d\Ha^{n-1}\;\le\; \frac{C_n}{h} 
\int_{\Eh_{t}\Delta\Eh_{t-h}}\vert \vh \vert\,dx.
\]

We now show how the above estimate implies \eqref{eq:L2-vh}. In view of 
\eqref{10} we have
\begin{eqnarray*}
\lefteqn{\int_{\partial \Eh_t}( \vh )^2\,d\Ha^{n-1}}\\
&\le&\frac{C_n}{h}\left( \Per(\Eh_{t-h}) + \frac{1}{\sqrt{h}}||\Eh_{t-h}|-1| - 
\Per(\Eh_{t}) - \frac{1}{\sqrt{h}}||\Eh_{t}|-1|\right).
\end{eqnarray*}
Integrating in time and using $|E_0|=1$, we obtain
\begin{eqnarray*}
\lefteqn{\int_0^{T}\int_{\partial \Eh_t}( \vh )^2\,d\Ha^{n-1}dt}\\
&\leq&C_n \left( \Per(\Eh_{0})  - \Per(\Eh_{T}) - 
\frac{1}{\sqrt{h}}||\Eh_{T}|-1|\right)
\\
&\leq& C_n \Per(E_{0}),
\end{eqnarray*}
from which, by a simple limit for $T \to +\infty$, \eqref{eq:L2-vh} follows.
\end{proof}

We now derive the Euler--Lagrange equations which constitute the weak motion law 
for the time-discrete evolution.

\begin{lemma}[(Euler--Lagrange equations)]\label{lem:EL}
For every $t\in[h,+\infty)$ and $\Psi\in C^1_c(\R^n;\R^n)$, it holds
\begin{equation}\label{eq:EL-sing-h}
\int_{\partial \Eh_t}\left(\mathrm{div}_{\partial  \Eh_t}\Psi+ \vh 
\nu_{\Eh_t}\cdot\Psi\right)\,d\Ha^{n-1}
\;=\;\lambda^{(h)}(t)
\int_{\partial \Eh_t}\nu_{\Eh_t}\cdot\Psi\,d\Ha^{n-1},
\end{equation}
where
\begin{equation}\label{eq1}
\lambda^{(h)}(t):=\frac{1}{\Ha^{n-1}(\partial \Eh_t)}
\int_{\partial \Eh_t}\left(H_{ \Eh_t}+ \vh\right)\,d\Ha^{n-1}.
\end{equation}
Moreover, if $\vert\Eh_t\vert\neq 1$, then it also holds
$\lambda^{(h)}(t)=\frac{1}{\sqrt{h}}\,\mathrm{sgn}(1-\vert\Eh_t\vert)$.
\end{lemma}

As we shall see in the proof below, the constants $\lambdah(t)$ defined in
\eqref{eq1} are Lagrange multipliers corresponding to the volume constraint,
whenever it is active. Since this constraint is satisfied up to a 
finite number of times (uniformly in $h$) by Corollary \ref{prop:soft2hard} 
below, by a slight abuse of terminology, we call $\lambdah(t)$ a {\em Lagrange 
multiplier}, even if the volume constraint is not active.

\begin{proof}
If $\vert\Eh_t\vert\neq 1$, it is very simple to compute the variations of 
$\Fh(\cdot,\Eh_{t-h})$ along the vector field $\Psi\in C^1_c(\R^n;\R^n)$
and see that they are given by \eqref{eq:EL-sing-h} 
with $\lambda^{(h)}(t)=\frac{1}{\sqrt{h}}\,\mathrm{sgn}(1-\vert\Eh_t\vert)$.
In the case $\vert\Eh_t\vert= 1$, we have 
\begin{gather*}
\Eh_t\in\argmin\left\{\Per(F)+\int_{F\Delta\Eh_{t-h}}\d_{ \Eh_{t-h}}\,dx:~\vert 
F\vert=1\right\}.
\end{gather*}
Hence,  performing variations of
$$
\Per(F)+\int_{F\Delta\Eh_{t-h}}\d_{ \Eh_{t-h}}\,dx
$$
within the class of sets of unit volume, for every $\Psi\in C^1_c(\R^n;\R^n)$, 
we again find \eqref{eq:EL-sing-h}, where $\lambda^{(h)}(t)$ is the Lagrange 
multiplier related 
to the constraint $\vert F\vert=1$. Observe that in both cases, 
we can  choose a sequence of
$\Psi\in C^1_c(\R^n;\R^n)$ approximating $\nu_{\Eh_t}$ 
on $\partial \Eh_t$, and conclude that $\lambda^{(h)}(t)$ is given by \eqref{eq1}.
\end{proof}

It is now clear that the regular part of $\de^* \Eh_t$ is of class $C^{2,\kappa}$. Indeed, by
choosing a suitable system of coordinates, $\de^* \Eh_t$ can be written in a neighbourhood
of any regular point as the graph of a $C^{1,\kappa}$ function $u$ solving the following equation
in the sense of distributions:
\[
\div\left( \frac{\nabla u}{\sqrt{1+|\nabla u|^2}}\right) = v^{(h)} - \lambda^{(h)}.
\]
Since $v^{(h)}$ is Lipschitz continuous, by standard elliptic regularity theory (cp., e.g., \cite{GilbargTrudinger}), 
one then deduce that $u \in C^{2,\kappa}$ for every $\kappa \in (0, 1)$.

\medskip

Next we prove that the whole family of sets defining the discrete flow up to time $T>0$ is contained in a large ball, 
whose radius does not depend on the discrete time-step $h$ but may depend on the $T$. 

\begin{lemma}[(Boundedness of minimizing sets)] \label{lem:ESP}
Let $\{\Eh_t\}_{t \geq 0}$ be an approximate solution and let $T>0$.
Then there exist $h_0, R_T>0$ (depending on $T$, $n$, and $E_0$ only)
such that, if $h\leq h_0$, then $\Eh_t\subset B_{R_T}$ for all $t\in[0,T]$.
\end{lemma}

\begin{proof} We fix $h>0$, and for every $t\in[0,T)$ we let
$$
r_{t}:=\inf\{r>0:~\Eh_t\subset B_r \}.
$$ 
We notice ${\bar B_{r_t}}\cap\partial\Eh_t\neq \emptyset$ is made of regular points
(because there are no singular minimizing cones contained in a half space, cp.~\cite[Theorem~15.5]{Giusti}),
and moreover
$$
{\bar B_{r_t}}\cap\partial\Eh_t\subset\left\{y\in\partial\Eh_t:\:H_{ \Eh_t}(y)\ge0\right\}.
$$
By this observation and the Euler--Lagrange equation
$\vh (t,y)= \lambdah(t) - H_{\Eh_t}(y)$, it follows that
\[
r_{t} \leq r_{t-h} + h\,|\lambdah(t)|.
\]
Iterating the above estimate, we then deduce that
\begin{equation}\label{20}
r_\tau \leq r_0 + \int_0^{\tau} |\lambdah(t)|\, dt \quad \forall\;\tau \in [0,T].
\end{equation}
To get some control on $\lambdah(t)$ we consider  $\Psi\in C_c^1(\R^n;\R^n)$  such that $\Psi(x)=x$ in $B_{r_t}$, 
and, since for $h$ small enough $|\Eh_t|\ge\frac12$ by \eqref{e:volume converging},
using $\Psi$ as test in \eqref{eq:EL-sing-h}, we obtain 
\begin{align}
\frac{n}2 |\lambdah(t)| &\leq
\left|\lambdah(t)\int_{\Eh_t} \div \Psi\, dx\right|
=\left| \lambdah(t)\int_{\partial\Eh_t} \nu_{\Eh_t}\cdot\Psi\, d\Ha^{n-1}\right|\notag\\
&= \left|\int_{\partial \Eh_t} \left(\div_{\partial \Eh_t} \Psi + \vh  \nu_{\Eh_t}\cdot\Psi\right)\, d\Ha^{n-1}\right|\notag\\
& \leq (n-1) \Per(\Eh_t) + r_t \Per(\Eh_t)^{1/2} \|\vh (t,\tacka)\|_{L^2(\partial \Eh_t)},\label{eq:boia-deh}
\end{align}
where we used $|\Psi|\le r_t$ on $\partial \Eh_t$.
Integrating in time and using the Cauchy--Schwarz inequality
together with \eqref{e:per decreasing} and \eqref{eq:L2-vh}, we obtain
\begin{equation}\label{21}
\int_0^{\tau} |\lambdah(t)|\, dt \leq  C_{n,0}{\tau} + C_{n,0}\left(\int_0^{\tau} r_t^2\, dt\right)^{1/2}.
\end{equation}
Combining \eqref{20} and \eqref{21}, it follows that
\begin{equation}\label{22}
r_{\tau}\leq r_0 + C_{n,0}{\tau} + C_{n,0}\left(\int_0^{\tau} r_t^2\, dt\right)^{1/2} \quad\mbox{for all 
$\tau\in[0,T]$}.
\end{equation}
The remainder of the proof is a standard ODE argument. Indeed, squaring both sides of the equation and redefining 
$C_{n,0}$ yields
\[
\frac{d}{d\tau} \left(e^{-C_{n,0}\tau}F(\tau)\right)\;\le\;C_{n,0} e^{-C_{n,0}\tau}\left(r_0^2 + 
\tau^2\right)\quad\mbox{for all $\tau\in[0,T]$},
\]
where $F(\tau) = \int_0^{\tau} r_t^2\, dt$. Integration in $\tau$ over the interval $[0,T]$ yields
\[
\int_0^{T} r_t^2\, dt\;\le\;C_{n,0,T} ,
\]
and thus the statement in Lemma \ref{lem:ESP} follows via \eqref{22}.
\end{proof}

\begin{corollary}\label{lem:bebysitter}
For every $h>0$ small enough, it holds
$$
\int_0^T \int_{\de \Eh_t}H^2_{\Eh_t}\, d\Ha^{n-1}dt+
\int_0^T\vert \lambda^{(h)}(t)\vert^2\,dt\;\leq\;C_{n,0,T}.
$$
\end{corollary}

\begin{proof}
The integrability of $\lambda^{(h)}$ follows from \eqref{eq:boia-deh} and from Lemma~\ref{lem:ESP};
the one of $H_{\Eh_t}$ follows taking into account the first variation \eqref{eq:EL-sing-h} and the integrability of the discrete velocity, Lemma \ref{Lem:vel-L^2}.
\end{proof}

For every $h>0$ we set
\begin{gather*}
\Sigma(h):=\{t:~ \vert \Eh_{t}\vert\neq 1\}.
\end{gather*}
\begin{corollary}\label{prop:soft2hard}
For every $h>0$ small enough, we have
$$
\vert\Sigma(h)\vert\;\leq\; C_{n,0,T}h.
$$
\end{corollary}

\begin{proof}
In view of Lemma \ref{lem:EL}, it is
$$
\Sigma(h)\subset\left\{ t\in[0,T):\vert\lambda^{(h)}(t)\vert\geq 1/\sqrt h\right\},
$$
and thus we have by Corollary~\ref{lem:bebysitter} 
\[
\vert\Sigma(h)\vert \;\leq \;h\int_0^T\vert\lambda^{(h)}(t)\vert^2\,dt\;\le\;C_{n,0,T}h. \qedhere
\]
\end{proof}

\subsection{Proof of Theorem~\ref{theo:conv-L1}}
The proof of the existence of a flat flow is now a simple
consequence of the results above.
Indeed, by \eqref{e:per decreasing}, \eqref{e:volume converging} and
Lemma~\ref{lem:ESP}, one can find sets $\{E_t\}_{t\in \Q^+}$, where $\Q^+$
denotes the set of positive rational numbers, and
a subsequence $h_k \downarrow 0$ such that
\[
\lim_{k\to +\infty}|E^{(h_k)}_t \Delta E_t| = 0 \quad  \forall\; t\in \Q^+.
\] 
Using the triangular inequality and Proposition~\ref{p:hoelder}, we deduce that
\begin{align}\label{e:hoelder cont}
|E_t\Delta E_s| &\leq \lim_{k\to+\infty}\Big( |E_t\Delta E^{(h_k)}_t| + |E^{(h_k)}_t\Delta E^{(h_k)}_s| + |E^{(h_k)}_s\Delta E_s|\Big)\notag\\
&\leq C_{n,0}|s-t|^{\sfrac12} \quad \forall\;0\leq s\leq t \in \Q^+.
\end{align}
Now a simple continuity argument implies that the sequence $E^{(h_k)}_t$ converges
to sets $E_t$ for all times $t\geq 0$ and satisfies \eqref{e:hoelder cont} for all $s,\, t\in[0,+\infty)$.
Finally, note that passing to the limit in \eqref{10} yields that $|E_t|=1$ and $\Per(E_t)\le\Per(E_0)$ (cf.\ 
\cite[Sec.\ 5.2.1]{EvansGariepy}) for a.e.\ $t\in(0,T)$.

\begin{remark}
It is also possible to show that the sequence of characteristic functions
\[
\chih(t,x):=\chi_{\Eh_t}(x)
\]
are precompact in $L^1((0,T)\times\R^n)$ for every $T>0$, thus giving an alternative proof of the theorem.
\end{remark}

%
%

\section{Distributional solutions}\label{s:4}
In this section we prove Theorem~\ref{theo:frenata} on the
existence of distributional solutions.
The two main ingredients of the proof besides the estimates of the previous
section are the hypothesis \eqref{eq:ass-LS} on the continuity
of the perimeters of the approximate solutions and the following proposition
which links the discrete velocity to the distributional time derivative
of the flat flow.

\begin{proposition}\label{p:time derivative}
Let $n\leq 7$ and $\{E^{(h)}_t\}_{t \geq 0}$ be 
an approximate solution to the volume-preserving mean-curvature flow.
Then, for every
$\phi \in C_c^\infty([0,+\infty) \times \R^n)$ it holds
\begin{equation}\label{e:time derivative}
\lim_{h \to 0}\Big\vert\int_h^{+\infty} \frac{1}{2h}
\Big[\int_{\Eh_t}\phi\,dx-\int_{\Eh_{t-h}}\phi\,dx\Big]dt
-\int_h^{+\infty}\int_{\partial\Eh_t}\phi \,v_{h}\,d\Ha^{n-1}dt\Big\vert
=0.
\end{equation}
\end{proposition}

Assuming the proposition we give a proof of the theorem.

\subsection{Proof of Theorem~\ref{theo:frenata}}
It follows straightforwardly from \eqref{eq:ass-LS} (cp., for instance,
\cite[Proposition~1.80]{AmFuPa}) that $\Ha^{n-1} \res \de^* E_t^{(h_k)}$ weakly
converges to $\Ha^{n-1} \res \de^* E_t^{}$ 
for almost every $t \in [0, +\infty)$.
In particular this implies that, for almost every $t \in [0, +\infty)$,
the boundaries $\de^* E_t^{(h_k)}$ converge to
$\de^* E_t$ in the sense of varifolds: namely, for a.e.~$t\in [0, +\infty)$
it holds
\begin{equation}\label{e:varifold conv}
\lim_{k\to \infty} \int_{\de^*\Ehk_t} F\big(x, 
\nu_{\Ehk_t}(x)\big)\,d\Ha^{n-1}(x)= 
\int_{\de^* E_t} F\big(x, \nu_{E_t}(x)\big)\,d\Ha^{n-1}(x),
\end{equation}
for every $F\in C_c(\R^n \times \R^n)$.
Indeed, by a simple approximation argument it is easy to verify
that it is enough to consider $F\in C_c^1(\R^n \times \R^n)$.
Then, for every $\eps>0$ we pick a continuous function $\nu_\eps :\R^n \to \R^n$
such that
\[
\int_{\de^* E_t} |\nu_{E_t} - \nu_\eps|^2 \,d\Ha^{n-1} \leq \eps^2,
\]
and estimate as follows
\begin{align*}
&\lim_{k\to+\infty}\left\vert\int_{\de^*\Ehk_t} \left(F\big(x,\nu_{\Ehk_t}(x)\big) - 
F\big(x,\nu_{\eps}(x)\big)\right)\,d\Ha^{n-1}(x)\right\vert\\
&\quad\leq \lim_{k\to+\infty} \|DF\|_{L^\infty}\,\int_{\de^* \Ehk_t} |\nu_{\Ehk_t} - \nu_\eps| \,d\Ha^{n-1}\\
&\quad\leq \lim_{k\to+\infty} \|DF\|_{L^\infty}\,\Per(\Ehk_t)^{\sfrac12}\left(\int_{\de^* \Ehk_t} |\nu_{\Ehk_t} - 
\nu_\eps|^2 \,d\Ha^{n-1}\right)^{\sfrac12}\\
& \quad= \|DF\|_{L^\infty}\,\Per(E_t)^{\sfrac12}\left(\int_{\de^* E_t} |\nu_{E_t} - \nu_\eps|^2 
\,d\Ha^{n-1}\right)^{\sfrac12}\\
&\quad \leq \|DF\|_{L^\infty}\,\,\Per(E_t)^{\sfrac12}\eps,
\end{align*}
where we used \eqref{eq:ass-LS} and the weak convergence of the the vector valued measures
$\nu_{\Ehk_t}\Ha^{n-1}\res \de^*\Ehk_t \weaks \nu_{E_t}\Ha^{n-1}\res \de^*E_t$ in the following way:
\begin{align*}
\lim_{k\to+\infty}\int_{\de^* \Ehk_t} &|\nu_{\Ehk_t} -\nu_\eps|^2 \,d\Ha^{n-1} \\
& = \lim_{k\to+\infty}
\int_{\de^* \Ehk_t} \big(1+ | \nu_\eps|^2 - 2 \nu_{\Ehk_t} \cdot \nu_{\eps}\big)\,d\Ha^{n-1}\\
& = \int_{\de^* E_t} \big(1+ | \nu_\eps|^2 - 2 \nu_{E_t} \cdot \nu_{\eps}\big)\,d\Ha^{n-1}\\
&= \int_{\de^* E_t} |\nu_{E_t} -\nu_\eps|^2 \,d\Ha^{n-1}.
\end{align*}

Next we use Lemma~\ref{Lem:vel-L^2} and Corollary~\ref{lem:bebysitter}
in conjunction with the results in Hutchinson \cite[Theorem~4.4.2]{Hutch}
to deduce the existence of 
functions $v:[0,+\infty) \times \R^n \to \R$, $\widehat\lambda:[0,+\infty) \to 
\R$ and $H:[0,+\infty) \times \R^n \to \R$ such that
\begin{gather*}
\int_0^T\vert\widehat\lambda\vert^2\,dt+\int_0^T\int_{\partial^*E_t}\left(\vert 
v\vert^2+|H|^2\right)\,\Ha^{n-1}dt<C_{n, 0,T},
\end{gather*}
and
\begin{gather}
\lim_{k\to\infty}\int_0^T\int_{\partial \Ehk_t}v_{h_k}\phi\,d\Ha^{n-1}dt=
\int_0^T\int_{\partial^*E_t}v\,\phi\,d\Ha^{n-1}dt,\label{e:conv1}\\
\lim_{k\to\infty}\int_0^T\int_{\partial \Ehk_t}\lambda^{(h_k)}\phi\,d\Ha^{n-1}dt=
\int_0^T\int_{\partial^*E_t}\widehat\lambda\,\phi\,d\Ha^{n-1}dt,\label{e:conv2}\\
\lim_{k\to\infty}\int_0^T\int_{\partial \Ehk}H_{E^{(h_k)}_t} 
\nu_{\Ehk}\cdot\Phi\,d\Ha^{n-1}dt=
\int_0^T\int_{\partial^*E_t}H\cdot\Phi\,d\Ha^{n-1}dt,\label{e:conv3}
\end{gather}
for every $\phi\in C^0_c([0,T)\times\R^n)$ and every 
$\Phi \in C^0_c([0,T)\times\R^n;\R^n)$.

\medskip

In particular, testing \eqref{e:varifold conv} with
$F(x, \nu) := \div \Psi - \nu \cdot \nabla \Psi \nu$ for some $\Psi \in C_c^1(\R^n;\R^n)$
and using \eqref{e:conv3},
by a simple approximation argument we conclude that, for a.e.~$t\in [0, +\infty)$,  
\begin{align}
\int_{\de^* E_t} \div_{\de E^t} \Psi \, d\Ha^{n-1} &=
\lim_{k \to +\infty} \int_{\de^* \Ehk_t} \div_{\de \Ehk_t} \Psi \, d\Ha^{n-1}\notag\\
& = \lim_{k\to +\infty} \int_{\de^* \Ehk_t} \nu_{\Ehk_t} \cdot \Psi \, H_{\Ehk_t}\,d\Ha^{n-1}\notag\\
& = \int_{\de^* E_t} \nu_{E_t} \cdot \Psi \, \notag H\,d\Ha^{n-1},
\end{align}
thus showing that $H(t, \cdot)$ is the generalized mean-curvature of $E_t$ for a.e.~$t \in [0,+\infty)$
(cp.~\eqref{e:generalized mean-curv}) and 
proving \eqref{e:gen mean-curv} of Theorem~\ref{theo:frenata}.

Similarly, \eqref{e:volMCF weak1} and \eqref{e:lagrange multip}
follows from \eqref{eq:EL-sing-h} and \eqref{eq1} by using 
\eqref{e:conv1} and \eqref{e:conv2}.

We need only to show \eqref{e:volMCF weak2}. To this aim we use Proposition~\ref{p:time derivative}.
For every $\phi \in C_c^1([0,+\infty) \times \R^n)$, by a change of variables we have that
\begin{multline*}
\int_h^\infty \Big[\int_{\Eh_t}\phi\,dx-\int_{\Eh_{t-h}}\phi\,dx\Big]dt  \\
= \int_h^\infty \int_{\Eh_t}\big(\phi(t, x) - \phi(t+h, x)\big)\,dx\,dt - h \int_{E_{0}}\phi\,dx,
\end{multline*}
where we used that $\Eh_t = E_0$ for $t\in [0, h)$.
Therefore it follows by a simple convergence argument that
\begin{align}
\lim_{h \to 0}\int_h^\infty \frac{1}{h} \Big[\int_{\Eh_t}\phi\,dx-\int_{\Eh_{t-h}}\phi\,dx\Big]dt  
= - \int_0^\infty \int_{E_t}\frac{\de \phi}{\de t}(t, x)\,dx\,dt - \int_{E_{0}}\phi\,dx.\notag
\end{align}
In view of \eqref{e:conv1} and \eqref{e:time derivative}, we conclude \eqref{e:volMCF weak2}
straightforwardly.

\subsection{Tilting of the tangent planes}
In this subsection and in the next one we give the proof of Proposition~\ref{p:time derivative}.
We follow closely the arguments in \cite{LuckhausSturzenhecker95} and for the sake
of completeness we provide a detailed proof in different steps.

This subsection is devoted to the estimate of the tilting of the normals
around points of small curvature. We recall that we assume
in this section $n\leq 7$ (in particular, the approximate solutions of 
the volume-preserving mean-curvature flow are everywhere of class $C^{2,\kappa}$).

\begin{lemma}\label{lem:rosticini}
For given constants $\frac12<\beta<\alpha<1$, there exists
a continuous increasing function $\omega:[0,1] \to \R$ with $\omega(0) =0$
with the following property.
Let $t\in[2h,+\infty)$, $\{\Eh_t\}_{t \geq 0}$ be an approximate
solution to the volume-preserving mean-curvature flow,
and let $x_0\in\partial\Eh_t$  be such that
\begin{equation}\label{e:small curvature}
\vert \vh (t,y)\vert\leq h^{\alpha-1}\quad \forall\; y\in B_{\gamma_n\sqrt{h}}(x_0)\cap (\Eh_t\Delta\Eh_{t-h}).
\end{equation}
Then there exists $\nu \in\R^n$ such that $\vert\nu\vert=1$ and
\begin{gather}
|\nu_{\partial\Eh_t}(y)-\nu| \leq \omega(h) \quad \forall\; y \in B_{h^\beta}(x_0) \cap \de \Eh_t,\label{e:constant1}\\
|\nu_{\partial\Eh_{t-h}}(y)-\nu| \leq \omega(h) \quad \forall\; y \in B_{h^\beta}(x_0) \cap \de 
\Eh_{t-h}.\label{e:constant2}
\end{gather}
\end{lemma}

\begin{proof}
Let $0<R \leq h^{\frac12-\beta}$ and
let $F\subset \R^n$ be any set such that $\Eh_t\Delta F\subset\subset \BRhbeta$.
By the minimizing property of $\Eh_t$ we have that
\begin{align}\label{e:minimizing}
\Per(\Eh_t,\BRhbeta) & \leq\Per(F,\BRhbeta) +\frac{1}{h}\int_{F\Delta \Eh_t}\d_{\Eh_{t-h}}(y)\,dy
\notag\\
&\quad+\frac{1}{\sqrt{h}} \big(\big\vert\vert F\vert-1\big\vert-\big\vert\vert \Eh_t\vert-1\big\vert\big).
\end{align}
A straightforward computation yields
\begin{gather*}
\big\vert\vert F\vert-1\big\vert-\big\vert \vert\Eh_t\vert-1\big\vert
\leq |F \Delta \Eh_t|,\\
\frac{1}{h}\int_{F\Delta \Eh_{t}}\d_{\Eh_{t-h}}(y)\,dy \leq  \frac{\gamma_n+1}{\sqrt{h}} \,|F \Delta \Eh_t|,
\end{gather*}
where we used that
$|\vh(t,y)|\le Rh^{\beta-1} + \gamma_n \, h^{-\sfrac12}\le (\gamma_n+1)\,h^{-\frac12}$ for all $y\in B_{Rh^\beta}(x_0) 
\cap(\Eh_t\Delta F)$ thanks to the fact that
$x_0 \in \de \Eh_t$, Proposition~\ref{p:Linfty} and the $1$-Lipschitz continuity
of the signed distance $\sd_{\Eh_{t-h}}$.
Combining the above estimates with \eqref{e:minimizing}, we obtain
\begin{equation}\label{e:almost minimizing}
\Per(E_t^{(h)},\BRhbeta)\;\leq\;\Per(F,\BRhbeta)+\frac{\gamma_n+2}{\sqrt{h}}\,|F \Delta \Eh_t|.
\end{equation}

Next we introduce the sets
\begin{align*}
\Ebetah_t&:=\Big\{z\in \R^n :~z=\frac{y-x_0}{h^\beta},\,y\in \Eh_t\Big\},\\
\Ebetah_{t-h}&:=\Big\{z\in \R^n :~z=\frac{y-x_0}{h^\beta},\,y\in \Eh_{t-h}\Big\}.
\end{align*}
By a simple rescaling argument, from \eqref{e:almost minimizing}
and from the analogous estimates at time $t-h$ (recall that $t\geq 2h$)
we deduce  that for $s=t, t-h$
\begin{gather}\label{e:lambda minimizing}
\Per(\Ebetah_s,B_R) \leq \Per(F, B_R) +(\gamma_n+2)\, h^{\beta-\frac{1}{2}}|F \Delta \Ebetah_s|\notag\\
\quad \forall\;R \leq h^{\sfrac12 - \beta},\quad \forall\; F \Delta \Ebetah_s \subset \subset B_R.\notag
\end{gather}
This implies that $\Ebetah_t$ and $\Ebetah_{t-h}$ are both $(\Lambda_h,r_h)$-minimizers of the perimeter on 
$\Lambda_h:=(\gamma_n+2) h^{\beta-\frac{1}{2}}$ and $r_h:=h^{\frac{1}{2}-\beta}$.
By the precompactness for sequences of 
$\Lambda_h$-minimizers (cf.\ \cite[Prop.\ 21.13]{Maggi}), 
we conclude that we can find a subsequence (not relabeled) verifying
\begin{gather*}
\lim_{h\to 0}\chi_{\Ebetah_{t}}=\chi_{E^{\beta}_1}\quad\text{and}\quad
\lim_{h\to0}\chi_{\Ebetah_{t-h}}=\chi_{E^{\beta}_{2}}\quad\text{in }L^1_{\loc}(\R^n).
\end{gather*}
Moreover, using the lower semicontinuity of the perimeter with respect to $L^1$ convergence and 
$\beta>\frac12$,
we deduce that $E^\beta_1$ and $E^\beta_2$ are locally minimizing the perimeter.
By the assumption $n\leq 7$ and a Bernstein theorem (see \cite[Theorem~17.3]{Giusti}),
$E^\beta_1,\, E^\beta_2$ are half-spaces. Moreover, 
by hypothesis it holds
\begin{gather*}
\d_{\Ebetah_{t-h}}(z) \leq
h^{\alpha-\beta}
\quad \forall\; z\in B_{h^{\frac{1}{2}-\beta}}(0)\cap \left(\Ebetah_t\Delta \Ebetah_{t-h}\right),
\end{gather*}
thus implying that $E^\beta_1= E^\beta_2$,
and by the fact that both are hyperplanes
there exists $\nu \in \R^n$ with $|\nu|=1$ such that
$$
E^\beta_1= E^\beta_2=\left\{z\in\R^n:~z\cdot \nu<0\right\}.
$$

To reach the conclusion of the lemma we need only to
invoke the regularity theory of $\Lambda$-minimizing set (cp.~\cite[Theorem~26.3]{Maggi})
and conclude that $\de E^{(h),\beta}_s$ is uniformly $C^{1,\kappa}$ in $B_{1}$ for $s= t, t-h$,
thus leading straightforwardly to \eqref{e:constant1} and \eqref{e:constant2}.
\end{proof}

\begin{corollary}\label{c:volume}
Under the hypotheses of Lemma~\ref{lem:rosticini}, let $\bC_{\sfrac{h^\beta}{2}}(x_0,\nu)\subset \R^n$
be the open cylinder defined as
\begin{multline}
\bC_{\sfrac{h^\beta}{2}}(x_0, \nu) \\
:= \left\{ x \in \R^n :~ |(x-x_0)\cdot \nu| <\sfrac{h^\beta}{2},\, 
\sqrt{|x-x_0|^2-|(x-x_0)\cdot \nu|^2} 
< \sfrac{h^\beta}{2} \right\}. \notag
\end{multline}
Then, there exists a dimensional constant $C>0$ such that
\begin{multline}\label{e:volume}
\left\vert \int_{\bC_{\sfrac{h^\beta}{2}}(x_0,\nu)} \big(\chi_{\Eh_t}-\chi_{\Eh_{t-h}}\big)\,dx 
- \int_{\de \Eh_t \cap \bC_{\sfrac{h^\beta}{2}}(x_0,\nu)}\sd_{\Eh_{t-h}} \, d\Ha^{n-1} \right\vert \\
\leq C\, \omega(h)
\int_{\bC_{\sfrac{h^\beta}{2}}(x_0,\nu)} \big\vert\chi_{\Eh_t}-\chi_{\Eh_{t-h}}\big\vert\,dx.
\end{multline}
\end{corollary}

\begin{proof}
From Lemma~\ref{lem:rosticini} we know that, for $h$ sufficiently small,
$\de \Eh_t$ and $\de \Eh_{t-h}$ in $\bC_{\sfrac{h^\beta}{2}}(x_0,\nu)$ can both be
written as graphs of functions of class $C^{1,\kappa}$.
Namely, by an affine change of coordinates we can assume without
loss of generality that $x_0 =0$ and $\nu = e_n$, and for simplicity
we set $\bC := \bC_{\sfrac{h^\beta}{2}}(0, e_n)$.
With this assumption we then have that for $s=t, t-h$
\[
\de \Eh_s \cap \bC= \left\{(y, f_s(y)) \in \R^{n-1}\times \R :~|y|\leq
\sfrac{h^{\beta}}{2}\right\},
\]
where $f_s:B_{\sfrac{h^\beta}{2}}\subset \R^{n-1} \to \R$ are $C^{1,\kappa}$ functions
with 
\begin{equation}\label{e:gradient}
\|\nabla f_s\|_{L^\infty(B_{\sfrac{h^\beta}{2}})} \leq \omega(h).
\end{equation}

In view of Fubini's theorem it is then clear that 
\begin{gather*}
\int_{\bC} \big(\chi_{\Eh_t}-\chi_{\Eh_{t-h}}\big)\,dx
= \int_{B_{\sfrac{h^\beta}{2}}} \big(f_t(y) - f_{t-h}(y)\big)\,dy,\\
\int_{\bC} \big\vert\chi_{\Eh_t}-\chi_{\Eh_{t-h}}\big\vert\,dx
= \int_{B_{\sfrac{h^\beta}{2}}} \big\vert f_t(y) - f_{t-h}(y)\big\vert\,dy.
\end{gather*}
Moreover, from \eqref{e:gradient} it follows that there exists a geometric
constant $C>0$ such that, for every $y \in B_{\sfrac{h^\beta}{2}}$,
\[
\big\vert \sd_{\Eh_{t-h}}(y,f_t(y))
 \sqrt{1+|\nabla f_t(y)|^2}
- (f_t(y) - f_{t-h}(y))\big\vert \leq C\, \omega(h)\,|f_t(y) - f_{t-h}(y)|.
\]
Therefore, one infers \eqref{e:volume} as follows
\begin{align*}
&\left\vert\int_{\de \Eh_t \cap \bC}\sd_{\Eh_{t-h}} \, d\Ha^{n-1}
- \int_{B_{\sfrac{h^\beta}{2}}} (f_t(y) - f_{t-h}(y))\,dy \right\vert
 \\
& \quad= \left\vert \int_{B_{\sfrac{h^\beta}{2}}} \Big(\sd_{\Eh_{t-h}}(y,f_t(y))
 \sqrt{1+|\nabla f_t(y)|^2}
-  (f_t(y) - f_{t-h}(y))\Big)\,dy \right\vert\\
& \quad \leq C\,\omega(h)\int_{B_{\sfrac{h^\beta}{2}}} |f_t - f_{t-h}|\,dy,
\end{align*}
where we used that \eqref{e:gradient}.
\end{proof}

We are finally ready for the proof of Proposition~\ref{p:time derivative}.

\subsection{Proof of Proposition~\ref{p:time derivative}}
We fix any time $t \in [2h, +\infty)$. For every $x_0 \in  \de \Eh_t$, we
fix $\alpha \in  \big(\frac12, \frac{n+2}{2(n+1)}\big)$ and consider the 
following open set $A_{x_0}$ defined as follows:
\begin{itemize}
\item[(i)] if \eqref{e:small curvature} holds,
then we set
$A_{x_0} := \bC_{\sfrac{h^\beta}{2}}(x_0, \nu)$ where $\nu\in \R^n$ is the unit
vector in Lemma~\ref{lem:rosticini};

\item[(ii)] otherwiese we set $A_{x_0} := B_{\gamma_n \sqrt{h}}(x_0)$.
\end{itemize}

Note that by Proposition~\ref{p:Linfty} we have that $\{A_{x_0}\}_{x_0 \in \de \Eh_t}$
is a covering of $\Eh_t \Delta \Eh_{t-h}$. Moreover, by a simple application of 
Besicovitch's covering theorem, cp.~\cite[Ch.\ 1.5.2]{EvansGariepy} (applied, for example, to
the balls to $B_{\sfrac{h^\beta}{2}}(x_0) \subset A_{x_0}$), there exists a finite
collections of points $I \subset \de \Eh_t$ such that $\{A_{x_0}\}_{x_0 \in I}$ is a 
covering of $\Eh_t \Delta \Eh_{t-h}$.

We estimate the contribution of the integrals in \eqref{e:time derivative} in every $A_{x_0}$
with $x_0 \in I$ in two steps, depending on whether (i) above applies or (ii).

\medskip

\textit{Estimate in case (i).} We use Corollary~\ref{c:volume} and deduce that
\begin{align}\label{e:estimate good set}
&\Big\vert \int_{A_{x_0}} \big(\chi_{\Eh_t} - \chi_{\Eh_{t-h}} \big) \, \phi\,dx - \int_{\de \Eh_t \cap A_{x_0}} 
\sd_{\Eh_{t-h}}\, \phi\, d\Ha^{n-1} \Big\vert \notag\\
&\quad \leq |\phi(x_0)| \Big\vert \int_{A_{x_0}} \big(\chi_{\Eh_t} - \chi_{\Eh_{t-h}} \big) \,dx - \int_{\de 
\Eh_t \cap A_{x_0}} \sd_{\Eh_{t-h}}\, d\Ha^{n-1} \Big\vert \notag\\
&\quad + \Big\vert \int_{A_{x_0}} \hspace{-0.2cm}\big(\chi_{\Eh_t} - \chi_{\Eh_{t-h}} \big) \, (\phi - \phi(x_0))\,dx - 
\int_{\de 
\Eh_t \cap A_{x_0}} \hspace{-0.2cm}
\sd_{\Eh_{t-h}}\, (\phi - \phi(x_0))\, d\Ha^{n-1} \Big\vert\notag\\
&\quad \leq C\,\big(\omega(h)\, \|\phi\|_{L^\infty} + h^\beta\, \|\nabla \phi\|_{L^\infty}\big) 
\int_{A_{x_0}} \big\vert\chi_{\Eh_t} - \chi_{\Eh_{t-h}} \big\vert\,d\Ha^{n-1},
\end{align}
where we used the fact that $A_{x_0} = \bC_{\sfrac{h^\beta}{2}}(x_0, \nu)$.

\medskip

\textit{Estimate in case (ii).} By assumption there exists a point 
$y_0 \in B_{\gamma_n\sqrt{h}}(x_0)\cap (\Eh_t\Delta\Eh_{t-h})$
such that $|v^{(h)}(t,y_0)|>h^{\alpha -1}$.
Without loss of generality we can assume that $y_0 \in \Eh_t$ (the other case can
be treated analogously and we leave the details to the reader).
It is then clear that $B_{\sfrac{h^{\alpha}}{2}}(y_0) \subset \R^n\setminus\Eh_{t-h}$
and $v^{(h)}(t,y)>\sfrac{h^{\alpha -1}}{2}$ for every $y 
\in B_{\sfrac{h^{\alpha}}{2}}(y_0)$.
Since $h^{\alpha} < 2\gamma_n  h^{\sfrac12}$, we can apply the density 
estimate in \eqref{e:density vol} and deduce that
\begin{gather} \label{e:from below}
C_n\,h^{(n+1)\alpha -1} \leq 
\int_{B_{\sfrac{h^{\alpha}}{2}}(y_0) \cap (\Eh_t \Delta \Eh_{t-h})} |v^{h}|\, dx.
\end{gather}
Similarly, by the density estimate in \eqref{e:density area} and Proposition~\ref{p:Linfty}
we deduce that
\begin{equation}\label{e:from above}
\int_{B_{\gamma_n\sqrt{h}}(x_0) \cap  \de \Eh_t } |\d_{E_{t-h}^{h}}|\, d\Ha^{n-1} \leq C_n\, h^{\sfrac{n}{2}}.
\end{equation}
From \eqref{e:from below}, \eqref{e:from above} and
$B_{\sfrac{h^{\alpha}}{2}}(y_0) \subset B_{2\gamma_n \sqrt{h}}(x_0)$ we then deduce that
\begin{align}\label{e:estimate bad set}
\int_{A_{x_0}} \big\vert \chi_{\Eh_t} - &\chi_{\Eh_{t-h}} \big\vert
+\int_{A_{x_0} \cap \de \Eh_{t}} \d_{\Eh_{t-h}} \,d\Ha^{n-1}\notag\\
&\leq C_n h^{\sfrac{n}{2} - (n+1)\alpha+1}
\int_{B_{2\,\gamma_n \sqrt{h}}(x_0) \cap (\Eh_t \Delta \Eh_{t-h})} |v^{h}|\, dx.
\end{align}

\medskip

We can then sum \eqref{e:estimate good set} and \eqref{e:estimate bad set} over $x_0 \in I$
and, recalling \eqref{10}, \eqref{e:per decreasing} and \eqref{e:volume converging},
we get
\begin{align}
&\Big\vert \int \big(\chi_{\Eh_t} - \chi_{\Eh_{t-h}} \big) \, \phi\,dx - \int_{\de \Eh_t} 
\sd_{\Eh_{t-h}}\, \phi\, d\Ha^{n-1} \Big\vert \notag\\
& \quad \leq \sum_{x_0\in I}\Big\vert \int_{A_{x_0}} \big(\chi_{\Eh_t} - \chi_{\Eh_{t-h}} \big) \, \phi\,dx - \int_{\de 
\Eh_t \cap A_{x_0}} 
\sd_{\Eh_{t-h}}\, \phi\, d\Ha^{n-1} \Big\vert\notag\\
& \quad \leq C_{n,0} \Big(\omega(h)\, \|\phi\|_{L^\infty} + h^\beta\, \|\nabla \phi\|_{L^\infty}
+ h^{\sfrac{n}{2} - (n+1)\alpha +1} \, \|\phi\|_{L^\infty}\Big) \notag\\
& \qquad \cdot \left(|\Eh_t \Delta \Eh_{t-h}| + \int_{\Eh_t \Delta \Eh_{t-h}} |v^{h}|\, dx\right) 
\end{align}
where we used the finite finiteness of the covering.

\medskip

Finally, integrating in time and using \eqref{e:extra1} and \eqref{e:extra2} we get
\begin{align*}
&\Big\vert\int_{2h}^{+\infty} \frac{1}{h}
\Big[\int_{\Eh_t}\phi\,dx-\int_{\Eh_{t-h}}\phi\,dx\Big]dt
-\int_h^{+\infty}\int_{\partial\Eh_t}\phi \,v_{h}\,d\Ha^{n-1}dt\Big\vert\notag\\
&\quad\leq C_{n,0,T} \Big(\omega(h)\, \|\phi\|_{L^\infty} + h^\beta\, \|\nabla \phi\|_{L^\infty}
+ h^{\sfrac{n}{2} - (n+1)\alpha +1} \, \|\phi\|_{L^\infty}\Big),
\end{align*}
where $T>0$ is such that $\textup{supp}(\phi) \subset [0,T] \times \R^n$.
Recalling the definition of $\alpha$ and taking the limit as $h$ goes to $0$, we conclude \eqref{e:time derivative}.

%
%
\appendix

\section{A density lemma}\label{s:a}
We premise the following density estimate for one-sided minimizers of the 
perimeter. The estimate can be easily deduced from the original arguments by De 
Giorgi exploited for minimizers \cite{DG}.

\begin{lemma}\label{l.one side}
There exists a dimensional constant $c_n>0$ with this property.
Let $E\subset \R^n$ be a set of finite perimeter, $R, \mu>0$ and 
$x_0\in \de E$ be such that
\begin{equation}\label{e.onesided}
\Per(E) \leq \Per (E\cup B_r(x_0)) + \mu \,|B_r(x_0) \setminus E|
\quad \forall \; 0<r < R.
\end{equation}
Then,
\begin{equation}\label{e.density3}
c_n\,r^n\leq |B_r(x_0)\setminus E| \quad \forall \; 0< r < 
\min\left\{R,\mu^{-1}\right\}.
\end{equation}
\end{lemma}

Recall that $x_0 \in \de E$ if $\min\{|B_r(x_0)\setminus E|,|E\cap B_r(x_0)|\}>0$ for every $r>0$.

\begin{proof}
Without loss of generality, we can assume $x_0 = 0$.
We use the following indentity which are true a.e.~$r>0$:
\begin{gather}
\Per(E\cup B_r)=\Ha^{n-1}(\de B_r\setminus E)+\Per(E,\R^n\setminus 
B_r(x)),\label{e:identity1}\\
\Per(B_r\setminus E)=\Ha^{n-1}(\de B_r\setminus 
E)+\Per(E,B_r),\label{e:identity2}\\
\Per(E)=\Per(E,B_r)+\Per(E,\R^n\setminus B_r)\label{e:identity3}.
\end{gather}
Indeed, if $E$ were smooth, these formulas follow for all the $r$ such that 
$B_r$ and $E$ have transversal intersections. Otherwise one can argue by 
approximation.
Using now \eqref{e.onesided}, we deduce that, for a.e.~$r>0$,
\begin{align}\label{e.half}
\Per(B_r\setminus E) & \stackrel{\eqref{e:identity2}}{=}
\Ha^{n-1}(\de B_r\setminus E)+\Per(E,B_r)\notag\\
&\stackrel{\eqref{e:identity3}}{\leq}
\Ha^{n-1}(\de B_r\setminus E)+\Per(E) - \Per(E,\R^n\setminus B_r)
\notag\\
&\stackrel{\eqref{e.onesided}\,\&\, \eqref{e:identity1}}{\leq}
2\,\Ha^{n-1}(\de B_r\setminus E) +\mu \,|B_r\setminus E|.
\end{align}
By the isoperimetric inequality \cite[Corollary~1.29]{Giusti}, there exists a 
dimensional constant $C>0$, such that
\begin{equation}\label{e.dif ineq}
C\,|B_r\setminus E|^{\frac{n-1}{n}}\leq \Per(B_r\setminus E) 
\stackrel{\eqref{e.half}}{\leq}  2\,\Ha^{n-1}(\de B_r\setminus E)
+ \mu \,|B_r \setminus E|.
\end{equation}

Setting $f(r):=|B_r\setminus E|$, by the coarea formula \cite[3.4.4]{EvansGariepy}, it 
holds
\[
\Ha^{n-1}(\de B_r\setminus E)=f'(r) \quad \text{for a.e. } r>0.
\]
Hence, \eqref{e.dif ineq} reads as
\begin{equation}\label{e:diff ineq}
C\,f(r)^{\frac{n-1}{n}}\leq 2\, f'(r) + \mu \, f(r).
\end{equation}
Finally, note that $f(r) \leq \omega_n\,r^n$, from which
$f(r) \leq \omega_n^{\sfrac{1}{n}} r\,f(r)^{\sfrac{n-1}{n}}$.
Therefore, there exists a dimensional constant $C_n>0$ such that
if $0< r < \min\{R, C_n\,\mu^{-1}\}$, then the last term in \eqref{e:diff ineq}
can be absorbed in the left hand side and deduce that
\[
f(r)^{\frac{n-1}{n}}\leq C\, f'(r).
\]
Integrating \eqref{e.dif ineq} (recall that $f(r)>0$ for every $r>0$)
we get the desired \eqref{e.density3}
for every $0< r < \min\{R, C_n\,\mu^{-1}\}$ and, by changing the
dimensional constant $c_n>0$, for every $0< r < \min\{R, \mu^{-1}\}$.
\end{proof}

%
%

\bibliographystyle{abbrv}
\bibliography{ms_lit}

\begin{thebibliography}{10}

\bibitem{AA}
M.~Alfaro and P.~Alifrangis.
\newblock Convergence of a mass conserving {A}llen-{C}ahn equation whose
  lagrange multiplier is nonlocal and local.
\newblock {\em Interfaces Free Bound.}, 16:243--268, 2014.

\bibitem{ATW}
F.~Almgren, J.~E. Taylor, and L.~Wang.
\newblock Curvature-driven flows: a variational approach.
\newblock {\em SIAM J. Control Optim.}, 31(2):387--438, 1993.

\bibitem{AmFuPa}
L.~Ambrosio, N.~Fusco, and D.~Pallara.
\newblock {\em Functions of bounded variation and free discontinuity problems}.
\newblock Oxford Mathematical Monographs. The Clarendon Press Oxford University
  Press, New York, 2000.

\bibitem{Andrews01}
B.~Andrews.
\newblock Volume-preserving anisotropic mean curvature flow.
\newblock {\em Indiana Univ. Math. J.}, 50(2):783--827, 2001.

\bibitem{BCCN}
G.~Bellettini, V.~Caselles, A.~Chambolle, and M.~Novaga.
\newblock The volume preserving crystalline mean curvature flow of convex sets
  in {$\mathbb{R}^N$}.
\newblock {\em J. Math. Pures Appl.}, 92(5):499--527, 2009.

\bibitem{BS}
L.~Bronsard and B.~Stoth.
\newblock Volume-preserving mean curvature flow as a limit of a nonlocal
  ginzburg-landau equation.
\newblock {\em SIAM J. Math. Anal.}, 28(4):769--807, 1997.

\bibitem{CR-M07}
E.~Cabezas-Rivas and V.~Miquel.
\newblock Volume preserving mean curvature flow in the hyperbolic space.
\newblock {\em Indiana Univ. Math. J.}, 56(5):2061--2086, 2007.

\bibitem{CR-M12}
E.~Cabezas-Rivas and V.~Miquel.
\newblock Volume preserving mean curvature flow of revolution hypersurfaces
  between two equidistants.
\newblock {\em Calc. Var. Partial Differential Equations}, 43(1-2):185--210,
  2012.

\bibitem{Caginalp89}
G.~Caginalp.
\newblock Stefan and hele-shaw type models as asymptotic limits of the
  phase-field equations.
\newblock {\em Phys. Rev. A}, 39(11):5887--5896, Jun 1989.

\bibitem{CapuzzoFinziMarch02}
I.~Capuzzo~Dolcetta, S.~Finzi~Vita, and R.~March.
\newblock Area-preserving curve-shortening flows: from phase separation to
  image processing.
\newblock {\em Interfaces Free Bound.}, 4(4):325--343, 2002.

\bibitem{CRCT95}
W.~Carter, A.~Roosen, J.~Cahn, and J.~Taylor.
\newblock Shape evolution by surface diffusion and surface attachment limited
  kinetics on completely faceted surfaces.
\newblock {\em Acta Metallurgica et Materialia}, 43(12):4309--4323, 1995.

\bibitem{CHL}
X.~Chen, D.~Hilhorst, and E.~Logak.
\newblock Mass conserving {A}llen-{C}ahn equation and volume preserving mean
  curvature flow.
\newblock {\em Interfaces Free Bound.}, 12(4):527--549, 2010.

\bibitem{DaiNiethammerPego}
S.~Dai, B.~Niethammer, and R.~L. Pego.
\newblock Crossover in coarsening rates for the monopole approximation of the
  mullins?sekerka model with kinetic drag.
\newblock {\em Proceedings of the Royal Society of Edinburgh: Section A
  Mathematics}, 140(03):553--571, 2010.

\bibitem{DG}
E.~De~Giorgi.
\newblock Nuovi teoremi relativi alle misure {$(r-1)$}-dimensionali in uno
  spazio ad {$r$} dimensioni.
\newblock {\em Ricerche Mat.}, 4:95--113, 1955.

\bibitem{EscherSimonett98}
J.~Escher and G.~Simonett.
\newblock The volume preserving mean curvature flow near spheres.
\newblock {\em Proc. Amer. Math. Soc.}, 126(9):2789--2796, 1998.

\bibitem{EvansGariepy}
L.~C. Evans and R.~F. Gariepy.
\newblock {\em Measure theory and fine properties of functions}.
\newblock Studies in Advanced Mathematics. CRC Press, Boca Raton, FL, 1992.

\bibitem{Gage86}
M.~Gage.
\newblock On an area-preserving evolution equation for plane curves.
\newblock In M.~Iannelli, R.~Nagel, and S.~Piazzera, editors, {\em Nonlinear
  problems in geometry (Mobile, Ala., 1985)}, volume~51 of {\em Contemp.
  Math.}, pages 52--62. Amer. Math. Soc., 1986.

\bibitem{GilbargTrudinger}
D.~Gilbarg and N.~S. Trudinger.
\newblock {\em Elliptic partial differential equations of second order}.
\newblock Classics in Mathematics. Springer-Verlag, Berlin, 2001.
\newblock Reprint of the 1998 edition.

\bibitem{Giusti}
E.~Giusti.
\newblock {\em Minimal surfaces and functions of bounded variation}, volume~80
  of {\em Monographs in Mathematics}.
\newblock Birkh\"auser Verlag, Basel, 1984.

\bibitem{GoldmanNovaga12}
M.~Goldman and M.~Novaga.
\newblock Volume-constrained minimizers for the prescribed curvature problem in
  periodic media.
\newblock {\em Calc. Var. Partial Differential Equations}, 44(3-4):297--318,
  2012.

\bibitem{Gurtin93}
M.~E. Gurtin.
\newblock {\em Thermomechanics of evolving phase boundaries in the plane}.
\newblock Oxford Mathematical Monographs. The Clarendon Press Oxford University
  Press, New York, 1993.

\bibitem{Huisken87}
G.~Huisken.
\newblock The volume preserving mean curvature flow.
\newblock {\em J. Reine Angew. Math.}, 382:35--48, 1987.

\bibitem{Hutch}
J.~E. Hutchinson.
\newblock Second fundamental form for varifolds and the existence of surfaces
  minimising curvature.
\newblock {\em Indiana Univ. Math. J.}, 35(1):45--71, 1986.

\bibitem{Li09}
H.~Li.
\newblock The volume-preserving mean curvature flow in {E}uclidean space.
\newblock {\em Pacific J. Math.}, 243(2):331--355, 2009.

\bibitem{LS61}
I.~M. Lifshitz and V.~V. Slyozov.
\newblock The kinetics of precipitation from supersaturated solid solutions.
\newblock {\em Journal of Physics and Chemistry of Solids}, 19(1-2):35--50,
  1961.

\bibitem{LuckhausSturzenhecker95}
S.~Luckhaus and T.~Sturzenhecker.
\newblock Implicit time discretization for the mean curvature flow equation.
\newblock {\em Calc. Var. Partial Differential Equations}, 3(2):253--271, 1995.

\bibitem{Maggi}
F.~Maggi.
\newblock {\em Sets of finite perimeter and geometric variational problems},
  volume 135 of {\em Cambridge Studies in Advanced Mathematics}.
\newblock Cambridge University Press, Cambridge, 2012.
\newblock An introduction to geometric measure theory.

\bibitem{MichorMumford06}
P.~W. Michor and D.~Mumford.
\newblock Riemannian geometries on spaces of plane curves.
\newblock {\em J. Eur. Math. Soc. (JEMS)}, 8(1):1--48, 2006.

\bibitem{MS13}
L.~Mugnai and C.~Seis.
\newblock On the {C}oarsening {R}ates for {A}ttachment-{L}imited {K}inetics.
\newblock {\em SIAM J. Math. Anal.}, 45(1):324--344, 2013.

\bibitem{PengMeriman}
D.~Peng, B.~Merriman, S.~Osher, H.~Zhao, and M.~Kang.
\newblock A pde-based fast local level set method.
\newblock {\em J. Comput. Phys.}, 155(2):410--438, 1999.

\bibitem{RuuthWetton03}
S.~J. Ruuth and B.~T.~R. Wetton.
\newblock A simple scheme for volume-preserving motion by mean curvature.
\newblock {\em J. Sci. Comput.}, 19(1-3):373--384, 2003.
\newblock Special issue in honor of the sixtieth birthday of Stanley Osher.

\bibitem{Tarshis}
L.~A. Tarshis, J.~L. Walker, and M.~F.~X. Gigliotti.
\newblock Solidification.
\newblock {\em Annual Review of Materials Science}, 2(1):181--216, 1972.

\bibitem{Wagner61}
C.~Wagner.
\newblock Theorie der {A}lterung von {N}iederschl\"agen durch {U}ml\"osen
  ({O}stwald-{R}eifung).
\newblock {\em Zeitschrift f\"ur {E}lektrochemie, {B}erichte der
  {B}unsengesellschaft f\"ur physikalische {C}hemie}, 65(7-8):581--591, 1961.

\end{thebibliography}

\end{document}